\documentclass[a4j,12pt]{article}
\usepackage{amsmath,amsthm,amssymb,amscd,ascmac, amsfonts}
\usepackage{mathrsfs}
\usepackage{stmaryrd}
\usepackage{braket}
\usepackage{accents}
\usepackage{url}
\allowdisplaybreaks[4]

\numberwithin{equation}{section}
\usepackage[dvips]{graphicx,color,psfrag}

\makeatletter

\renewcommand{\det}{\mathop{\mathrm{det}}\,}

\newtheorem*{multitheorem}{\variable@name}

\theoremstyle{definition}
\newcommand{\variable@name}{Theorem}
\newtheorem*{multiproclaim}{\variable@name}

\theoremstyle{plain}
\newtheorem{thm}{Theorem}
\newtheorem{prop}[thm]{Proposition}
\newtheorem{lem}[thm]{Lemma}
\newtheorem{cor}[thm]{Corollary}

\theoremstyle{definition}

\newtheorem{rem}[thm]{Remark}
\newtheorem{exa}[thm]{Example}

\begin{document}
\title{New identities for some symmetric polynomials, and a higher order analogue of the Fibonacci and Lucas numbers}
\author{Genki Shibukawa}
\date{
\small MSC classes\,:\,05E05, 11B39, 33C05}
\pagestyle{plain}

\maketitle

\begin{abstract}
We give new identities for some symmetric polynomials. 
As applications of these identities, we obtain some formulas for a higher order analogue of Fibonacci and Lucas numbers.  
\end{abstract}

\section{Introduction}
Throughout the paper, we denote the set of non-negative integers by $\mathbb{Z}_{\geq 0}$, 
the field of real numbers by $\mathbb{R}$, 
the field of complex numbers by $\mathbb{C}$ and $\mathbb{C}^{\times }:=\mathbb{C}\setminus \{0\}$. 
Let $z_{1},\ldots ,z_{r}$ be $r$ independent variables and $\mathbf{z}:=(z_{1},\ldots,z_{r})$. 
For each non-negative integer $n$, the $n$th elementary symmetric polynomial $e_{n}^{(r)}$, 
complete homogeneous polynomials $h_{n}^{(r)}$ and power symmetric polynomial $p_{n}^{(r)}$ are defined by 
\begin{align}
\label{eq:elementary sum}
e_{n}^{(r)}
   &=
   e_{n}^{(r)}(\mathbf{z})
   :=\begin{cases}
   \sum_{1\leq j_{1}<\cdots <j_{n}\leq r}
   z_{j_{1}}\cdots z_{j_{n}} & (1\leq n\leq r) \\
   1 & (n=0) \\
   0 & (n>r)
   \end{cases}, \\
\label{eq:complete homogeneous}
h_{n}^{(r)}
   &=
   h_{n}^{(r)}(\mathbf{z})
   :=
   \sum_{m_{1}+ \cdots+ m_{r}=n}
   z_{1}^{m_{1}}\cdots z_{r}^{m_{r}}, \\
\label{eq:power sum}
p_{n}^{(r)}
   &=
   p_{n}^{(r)}(\mathbf{z})
   :=
   \sum_{j=1}^{r}z_{j}^{n}
\end{align}
respectively. 
For $\mathbf{z} =(z _{1},\ldots, z _{r}) \in {\mathbb{C}^{\times }}^{r}$, we put
\begin{align}
(\mathbf{z}+\mathbf{z}^{-1})
   &:=
   (z_{1}+z_{1}^{-1},\ldots ,z_{r}+z_{r}^{-1})\in {\mathbb{C}}^{r}, \nonumber \\
(\mathbf{z},\mathbf{z}^{-1})
   &:=
   (z_{1},\ldots ,z_{r},z_{1}^{-1},\ldots ,z_{r}^{-1})\in {\mathbb{C}}^{2r}. \nonumber
\end{align}
Our main results are new identities for these three types of symmetric polynomials $f=e,h,p$, which are relationships between $f_{n}^{(r)}(\mathbf{z}+\mathbf{z}^{-1})$ and $f_{n}^{(2r)}(\mathbf{z},\mathbf{z}^{-1})$. 
More precisely, we determine the following expansion coefficients $a_{n,k}^{(f)}$ and $b_{n,k}^{(f)}$, 
\begin{align}
\label{eq:first kind}
f_{n}^{(r)}(\mathbf{z}+\mathbf{z}^{-1})
   &=
   \sum_{k=0}^{n}a_{n,k}^{(f)}f_{k}^{(2r)}(\mathbf{z},\mathbf{z}^{-1}), \\
\label{eq:second kind}
f_{n}^{(2r)}(\mathbf{z},\mathbf{z}^{-1})
   &=
   \sum_{k=0}^{n}b_{n,k}^{(f)}f_{k}^{(r)}(\mathbf{z}+\mathbf{z}^{-1}).  
\end{align}
In this article we call (\ref{eq:first kind}) (resp. (\ref{eq:second kind})) the first kind formulas (resp. the second kind formulas). 
\begin{thm}[The first kind formulas]
\label{thm:the first kind}
For any non-negative integer $m$, we have the following identities. \\
{\rm{(1)}} 
\begin{align}
\label{eq:main elemental}
   \sum_{k=\max\left(\left\lfloor \frac{m}{2}\right\rfloor -r,0\right)}^{\left\lfloor \frac{m}{2}\right\rfloor}
   c_{m-r -1,k}
   e_{m-2k}^{(2r)}(\mathbf{z},\mathbf{z}^{-1})
   =
   \begin{cases}
   e_{m}^{(r)}(\mathbf{z}+\mathbf{z}^{-1}) & (m=0,1,\ldots, r) \\
   0 & (\rm{others})
   \end{cases},
\end{align}
where 
$$
c_{n,k}:=\binom{n}{k}-\binom{n}{k-1}, \quad 
\binom{n}{k}:=\begin{cases}
   \frac{n(n-1)\cdots (n-k+1)}{k!} & (k\not=0) \\
   1 & (k=0) 
   \end{cases}.
$$
and $\lfloor x \rfloor$ is the greatest integer not exceeding $x \in \mathbb{R}$.\\
{\rm{(2)}}
\begin{align}
\label{eq:main complete}
h_{m}^{(r)}(\mathbf{z}+\mathbf{z}^{-1})
   =
   \sum_{k= 0}^{\left\lfloor \frac{m}{2}\right\rfloor}
   c_{m+r-1,k}
   h_{m-2k}^{(2r)}(\mathbf{z},\mathbf{z}^{-1}).
\end{align}
{\rm{(3)}} 
\begin{align}
\label{eq:main power}
p_{m}^{(r)}(\mathbf{z}+\mathbf{z}^{-1})
   =
   \frac{1}{2}
   \sum_{k=0}^{m}
   \binom{m}{k}
   p_{|m-2k|}^{(2r)}(\mathbf{z},\mathbf{z}^{-1}).
\end{align}
\end{thm}
\begin{thm}[The second kind formulas]
\label{thm:the second kind}
{\rm{(1)}} For $n=0,1,\ldots, 2r$, we have 
\begin{align}
\label{eq:main elemental2}
e_{n}^{(2r)}(\mathbf{z},\mathbf{z}^{-1})
   &=
   \sum_{k=\max\left\{\left\lfloor \frac{n-r}{2}\right\rfloor,0\right\}}^{\left\lfloor \frac{n}{2}\right\rfloor}
   \binom{r-n+2k}{k}
   e_{n-2k}^{(r)}(\mathbf{z}+\mathbf{z}^{-1}).
\end{align}
{\rm{(2)}} For any non-negative integer $n$, 
\begin{align}
\label{eq:main complete2}
h_{n}^{(2r)}(\mathbf{z},\mathbf{z}^{-1})
   &=
   \sum_{k=0}^{\left\lfloor \frac{n}{2}\right\rfloor}
   \binom{n-k+r-1}{k}
   h_{n-2k}^{(r)}(\mathbf{z}+\mathbf{z}^{-1}).
\end{align}
{\rm{(3)}} For any positive integer $n$, 
\begin{align}
\label{eq:main power2}
p_{n}^{(2r)}(\mathbf{z},\mathbf{z}^{-1})
   &=
   2\sum_{k= 0}^{\lfloor \frac{n+1}{2}\rfloor}\binom{2k-n-1}{k}
   p_{n-2k}^{(r)}(\mathbf{z}+\mathbf{z}^{-1})
   -\sum_{k= 0}^{\lfloor \frac{n}{2}\rfloor}\binom{2k-n}{k}p_{n-2k}^{(r)}(\mathbf{z}+\mathbf{z}^{-1}). 
\end{align}
\end{thm}
The proofs of (\ref{eq:main elemental}) and (\ref{eq:main complete}) are more difficult than (\ref{eq:main elemental2}) and (\ref{eq:main complete2}). 
In particular we need a hypergeometric identity (\ref{eq:key formula}) to derive the explicit formulas of $a_{n,k}^{(e)}$ or $a_{n,k}^{(h)}$. 
The second kind formulas and (\ref{eq:main power}) which are proved immediately by the binomial formula and generating functions may be well known. 
However we have not found appropriate references about these formulas and their interesting applications which are to propose some new formulas of Fibonacci and Lucas numbers etc. 
In particular, we apply Theorem \ref{thm:the first kind} and Theorem \ref{thm:the second kind} to some special values of $h_{n}^{(r)}$ and $p_{n}^{(r)}$;
\begin{align}
F_{n+1}^{(r)}
   &:=
   h_{n}^{(r)}\left(-\boldsymbol{\zeta}^{+\iota }-\boldsymbol{\zeta}^{-\iota }\right)
   =
   h_{n}^{(r)}\left(-2\cos{\left(\frac{2\pi }{2r+1}\right)},\ldots, -2\cos{\left(\frac{2\pi r}{2r+1}\right)}\right), \\
L_{n}^{(r)}
   &:=
   p_{n}^{(r)}\left(-\boldsymbol{\zeta}^{+\iota }-\boldsymbol{\zeta}^{-\iota }\right)   
   =
   p_{n}^{(r)}\left(-2\cos{\left(\frac{2\pi }{2r+1}\right)},\ldots, -2\cos{\left(\frac{2\pi r}{2r+1}\right)}\right),   
\end{align}
where  
$$
\boldsymbol{\zeta}^{\pm \iota }
   :=
   \left(e^{\pm \frac{2\pi \sqrt{-1}r}{2r+1}},\ldots , e^{\pm \frac{2\pi \sqrt{-1}}{2r+1}}\right).
$$
Since the case of $r=1$ is trivial sequences
\begin{align*}
F_{n+1}^{(1)}=L_{n}^{(1)}=1, \quad (n\geq 0)
\end{align*}
and the case of $r=2$ is the classical Fibonacci numbers $\{F_{n+1}\}_{n}$ and Lucas numbers $\{L_{n}\}_{n}$
\begin{align*}
F_{n+1}^{(2)}
   &=
   \frac{1}{\sqrt{5}} \left(\left(\frac{1+\sqrt{5}}{2}\right)^{n+1}\!\!\!\!-\left(\frac{1-\sqrt{5}}{2}\right)^{n+1}\right)
   =:
   F_{n+1}, \\  
L_{n}^{(2)}
   &=
   \left(\frac{1+\sqrt{5}}{2}\right)^{n+1}\!\!\!\!+\left(\frac{1-\sqrt{5}}{2}\right)^{n+1}
   =:
   L_{n},
\end{align*}
the sequences $\{F_{n+1}^{(r)}\}_{n}$ and $\{L_{n}^{(r)}\}_{n}$ are regarded as a higher order analogue of the classical Fibonacci and Lucas numbers. 
From Theorem \ref{thm:the first kind}, Theorem \ref{thm:the second kind} and some fundamental results for symmetric functions, we propose the following fundamental formulas for $F_{n+1}^{(r)}$ and $L_{n}^{(r)}$. 
\begin{thm}[Explicit formulas of $F_{n}^{(r)}$ and $L_{n}^{(r)}$]
\label{thm:explicit formulas}
\begin{align}
F_{n}^{(r)}
   &=
   \sum_{k= 0}^{\left\lfloor \frac{n-1}{2}\right\rfloor}
   \frac{1}{2}\left((-1)^{\left\lfloor \frac{n-1-2k}{2r+1}\right\rfloor}-(-1)^{\left\lfloor \frac{n-2k-3}{2r+1}\right\rfloor}\right)
   c_{n+r-2,k} \nonumber \\
   &=
   \sum_{k= 0}^{\left\lfloor \frac{n-1}{2r+1}\right\rfloor}
   (-1)^{k}
   c_{n+r-2,\left\lfloor \frac{n-1-(2r+1)k}{2}\right\rfloor}, \\
L_{n}^{(r)}
   &=
   (-1)^{n+1}2^{n-1}
   +(-1)^{n}\frac{2r+1}{2}\sum_{k=0}^{n}
   \binom{n}{k}
   \delta _{2r+1 \mid n-2k} \nonumber \\
   &=\begin{cases}
   -2^{2m-1}
   +\frac{2r+1}{2}\sum_{k=-\left\lfloor \frac{m}{2r+1}\right\rfloor}^{\left\lfloor \frac{m}{2r+1}\right\rfloor}
   \binom{2m}{m-(2r+1)k} & (n=2m) \\
   4^{m}
   -\frac{2r+1}{2}\sum_{k=-\left\lfloor \frac{m+r+1}{2r+1}\right\rfloor}^{\left\lfloor \frac{m-r}{2r+1}\right\rfloor}
   \binom{2m+1}{m-(2r+1)k-r} & (n=2m+1)
   \end{cases}, 
\end{align}
where
$$
\delta _{2r+1 \mid n-2k}
   :=
   \begin{cases}
   0 & (2r+1 \not| \,\,\, n-2k)\\
   1 & (2r+1 \mid n-2k)
   \end{cases}.
$$
\end{thm}
\begin{thm}[Inversion of the explicit formulas]
\label{thm:inversion of explicit formulas}
\begin{align}
& \sum_{k=0}^{\left\lfloor \frac{n}{2}\right\rfloor}
   \binom{n-k+r-1}{k}
   F_{n-2k+1}^{(r)} \nonumber \\
   &=
   \frac{1}{2}\left((-1)^{\left\lfloor \frac{n-1-2k}{2r+1}\right\rfloor}-(-1)^{\left\lfloor \frac{n-2k-3}{2r+1}\right\rfloor}\right)
   =\begin{cases}
   1 & (n\equiv 0,1\,(\mathrm{mod}\,4r+2)) \\
   -1 & (n\equiv 2r+1,2r+2\,(\mathrm{mod}\,4r+2)) \\
   0 & ({\rm{others}})
   \end{cases}, \\
& 2\sum_{k= 0}^{\lfloor \frac{n+1}{2}\rfloor}\binom{2k-n-1}{k}
   L_{n-2k}^{(r)}
   -\sum_{k= 0}^{\lfloor \frac{n}{2}\rfloor}\binom{2k-n}{k}L_{n-2k}^{(r)} \nonumber \\
   &=
   (-1)^{n}(-1+(2r+1)\delta_{2r+1\mid n})
   =
   \begin{cases}
   (-1)^{n-1} & (2r+1\not| \,\,\,n) \\
   (-1)^{n}2r & (2r+1\mid n)
   \end{cases}.     
\end{align}
\end{thm}
\begin{thm}[Initial values and recurrence formulas]
\label{thm:initial values and recurrence formulas}
Initial values of $\{F_{n+1}^{(r)}\}_{n \in \mathbb{Z}}$ and $\{L_{n}^{(r)}\}_{n \in \mathbb{Z}}$ are given by 
\begin{align}
&F_{1}^{(r)}=1, \quad F_{0}^{(r)}=F_{-1}^{(r)}=\cdots =F_{-r+2}^{(r)}=0, \\
&L_{n}^{(r)}
   =\begin{cases}
   -2^{2m-1}+\frac{2r+1}{2}\binom{2m}{m} & (n=2m) \\
   4^{m} & (n=2m+1) 
   \end{cases} \quad 
   (n=0,1,\ldots,r-1). 
\end{align}
The sequences $\{F_{n+1}^{(r)}\}_{n \in \mathbb{Z}}$ and $\{L_{n}^{(r)}\}_{n \in \mathbb{Z}}$ satisfy the following same recursion:
\begin{equation}
\label{eq:recursion of F and L}
a_{n+r}^{(r)}
   =
   \sum_{j=0}^{\left\lfloor \frac{r-1}{2}\right\rfloor}
   (-1)^{j}\binom{r-1-j}{j}a_{n+r-1-2j}^{(r)}
   +\sum_{j=0}^{\left\lfloor \frac{r-2}{2}\right\rfloor}
   (-1)^{j}\binom{r-1-j}{j+1}a_{n+r-2-2j}^{(r)}.
\end{equation}
\end{thm}

The content of this article is as follows. 
In Section\,2, we refer some basic formulas for symmetric polynomials and the Gauss hypergeometric function from \cite{AAR} and \cite{M}. 
Section\,3 is the main part of this article. 
In this section, we prove Theorem \ref{thm:the first kind} and Theorem \ref{thm:the second kind}, and give their principal specialization.  
In Section\,4, we evaluate $e_{n}^{(2r)}\left(-\boldsymbol{\zeta}^{+\iota },-\boldsymbol{\zeta}^{-\iota }\right)$, $h_{n}^{(2r)}\left(-\boldsymbol{\zeta}^{+\iota },-\boldsymbol{\zeta}^{-\iota }\right)$ and $p_{n}^{(2r)}\left(-\boldsymbol{\zeta}^{+\iota },-\boldsymbol{\zeta}^{-\iota }\right)$ by the definitions and generating functions. 
By substituting these evaluations into Theorem \ref{thm:the first kind} and Theorem \ref{thm:the second kind}, we propose Theorem \ref{thm:explicit formulas}, Theorem \ref{thm:inversion of explicit formulas} and Theorem \ref{thm:initial values and recurrence formulas}. 
Further, we consider some specializations of our main results and drive some interesting binomial sum formulas, including new formulas for the Fibonacci and Lucas numbers (see Corollary \ref{thm:new Fibonacci formulas} and Corollary \ref{thm:new Lucas formulas}). 
In Appendix, we mention some congruence properties and other formulas (generating functions, determinant formulas, some relations) for $F_{n+1}^{(r)}$ and $L_{n}^{(r)}$ that are proven independently of Theorem \ref{thm:the first kind} and Theorem \ref{thm:the second kind}. 

\section{Preliminaries}
\subsection{Symmetric polynomials}
Refer to Macdonald \cite{M} for the details in this subsection. 
We fix a positive integer $r$, and denote the partition set of length $r$ by 
$$
\lambda \in \mathcal{P}_{r}:=\{\nu =(\nu _{1},\ldots ,\nu _{r}) \in \mathbb{Z}_{\geq 0}^{r} \mid \nu _{1}\geq \ldots \geq \nu _{r}\}
$$
and the symmetric group of degree $r$ by $\mathfrak{S}_{r}$. 
For some special partitions, we use the following notations 
\begin{align}
(n)&:=(n,0,\ldots,0) \in \mathcal{P}_{r}, \nonumber \\
(1^{n})&:=(1,\ldots,1,0,\ldots,0) \in \mathcal{P}_{r}, \quad (n=1,\ldots ,r). \nonumber
\end{align}
The symmetric group $\mathfrak{S}_{r}$ acts on $\mathbf{z}=(z_{1},\ldots,z_{r}) \in \mathbb{C}^{r}$ by 
$$
\sigma .\mathbf{z}:=(z_{\sigma (1)},\ldots ,z_{\sigma (r)}). 
$$

For any partition $\lambda $, we define Schur polynomial $s_{\lambda }(\mathbf{z})$ and monomial symmetry polynomial $m_{\lambda }(\mathbf{z})$ by 
\begin{align}
\label{eq:def of Schur}
s_{\lambda }(\mathbf{z})
   &:=
      \frac{\det{\left(z_{i}^{\lambda _{j}+r-j}\right)}_{i,j=1,\ldots r}}{\det{\left(z_{i}^{r-j}\right)}_{i,j=1,\ldots r}}
   =
      \frac{\det{\left(z_{i}^{\lambda _{j}+r-j}\right)}_{i,j=1,\ldots r}}{\prod_{1\leq i<j\leq r}(z_{i}-z_{j})}, \\
m_{\lambda }(\mathbf{z})
   &:=
      \sum_{\nu \in \mathfrak{S}_{r}.\lambda }z_{1}^{\nu _{1}}\cdots z_{r}^{\nu _{r}},
\end{align}
where $\det$ is the usual determinant and 
$$
\mathfrak{S}_{r}.\lambda 
   :=\{\sigma .\lambda :=(\lambda _{\sigma (1)},\ldots ,\lambda _{\sigma (r)})\mid \sigma \in \mathfrak{S}_{r}\}. 
$$
We remark that Schur polynomial extends to Schur function $s_{\lambda }(\mathbf{z})$ for $\lambda =(\lambda _{1},\ldots ,\lambda _{r}) \in \mathbb{C}^{r}$ by (\ref{eq:def of Schur}).

It is well known that 
\begin{align}
e_{n}^{(r)}(\mathbf{z})
   &=
      s_{(1^{n})}(\mathbf{z})
      =
      m_{(1^{n})}(\mathbf{z}), \\
p_{n}^{(r)}(\mathbf{z})
   &=
      m_{(n)}(\mathbf{z}), \\
\label{eq:Schur and complete homogeneous}
h_{n}^{(r)}(\mathbf{z})
   &=
   s_{(n)}(\mathbf{z}),
\end{align}
which is Schur $s_{(1^{n})}$ or monomial $m_{(1^{n})}$ with one column are $e_{n}^{(r)}$, and monomial $m_{(n)}$ with one row is $p_{n}^{(r)}$, and Schur $s_{(n)}$ with one row is $h_{n}^{(r)}$ (\cite{M} Chapter\,I Section\,3 (3.9)). 
From (\ref{eq:Schur and complete homogeneous}), we extend the complete homogeneous polynomials to $h_{n}^{(r)}(\mathbf{z})$ $(n\in \mathbb{Z})$: namely 
\begin{equation}
\label{eq:negative complete homogeneous}
h_{-n}^{(r)}(\mathbf{z})
   :=
   s_{(-n)}(\mathbf{z}) \quad (n\geq 0).
\end{equation}
By this extension (\ref{eq:negative complete homogeneous}) and the definition of Schur function, for any $r$ we have
\begin{align}
\label{eq:vanishing property}
h_{-n}^{(r)}(\mathbf{z})=0 \quad (n=1,2,\ldots ,r-1). 
\end{align}

We list up some required formulas for symmetric polynomials in \cite{M}. 
\begin{lem}
{\rm{(1)}} Generating functions
\begin{align}
\label{eq:gen fnc of elemental}
\prod_{j=1}^{r}(1+z_{j}y)
   =\sum_{n=0}^{r}
      e_{n}^{(r)}(\mathbf{z})y^{n}, \\
\label{eq:gen fnc of complete}      
\prod_{j=1}^{r}\frac{1}{1-z_{j}y}
   =\sum_{n=0}^{\infty}
      h_{n}^{(r)}(\mathbf{z})y^{n}, \\
\label{eq:gen fnc of power}
\sum_{j=1}^{r}\frac{1}{1-z_{j}y}
   =\sum_{n=0}^{\infty}
      p_{n}^{(r)}(\mathbf{z})y^{n}.
\end{align}
{\rm{(2)}} $q$-binomial formula 
\begin{align}
\label{eq:q-binom}
\prod_{j=0}^{n-1}(1+q^{j}y)
   &=
   \sum_{k=0}^{n}\binom{n}{k}_{q}q^{\frac{k(k-1)}{2}}y^{k}, \\
   \label{eq:q-binom2}
\prod_{j=0}^{n-1}\frac{1}{1-q^{j}y}
   &=
   \sum_{k=0}^{\infty}\binom{n+k-1}{k}_{q}y^{k},
\end{align}
where $\binom{n}{k}_{q}$ denotes the $q$-binomial coefficient
$$
\binom{n}{k}_{q}:=\frac{(1-q^{n})\cdots (1-q^{n-k+1})}{(1-q)\cdots (1-q^{k})}.
$$
{\rm{(3)}} Wronski relation and Newton's formula 
\begin{align}
\label{eq:Wronski}
\sum_{j=0}^{\min{(n,r)}}(-1)^{j-1}e_{j}^{(r)}(\mathbf{z})h_{n-j}^{(r)}(\mathbf{z})
   &=
   0, \\
\label{eq:Newton}
\sum_{j=0}^{\min{(n,r)}}(-1)^{n-j}p_{n-j+1}^{(r)}(\mathbf{z})e_{j}^{(r)}(\mathbf{z})
   &=
   (n+1)e_{n+1}^{(r)}(\mathbf{z}). 
\end{align}
\end{lem}
Actually, (\ref{eq:gen fnc of elemental}) is \cite{M} p19 (2.2) and (\ref{eq:gen fnc of complete}) is \cite{M} p21 (2.5) exactly. 
For (\ref{eq:q-binom}) and (\ref{eq:q-binom2}), see \cite{M} p26 Examples\,3. 
Similarly, (\ref{eq:Wronski}) and (\ref{eq:Newton}) are \cite{M} p21 (2.6${}^{\prime}$) and p23 (2.11${}^{\prime}$) respectively. 

\subsection{The Gauss hypergeometric function}
Let $a,b,c,z$ be complex numbers such that $c$ is not non-negative integers, 
and $(a)_{m}$ be the raising factorial defined by 
$$
(a)_{m}
   :=
   \begin{cases}
   a(a+1)\cdots (a+m-1) & (m\not=0) \\
   1 & (m=0)
   \end{cases}.
$$
We recall Gauss hypergeometric function
$$
{_{2}F_1}(z)
   =
   {_{2}F_1}\left(\begin{matrix}a,b \\ c \end{matrix};z\right)
   :=
   \sum_{m= 0}^{\infty}
      \frac{(a)_{m}(b)_{m}}{m!(c)_{m}}z^{m} \quad (|z|<1),
$$
and for any complex numbers $\alpha $ and $x$ we put
\begin{equation}
\psi(\alpha ;x)
   :=
   \sum_{k=0}^{\infty}
      \frac{(\alpha )_{2k}}{k!(\alpha +1)_{k}}x^{k} \quad (|4x|<1). 
\end{equation}
Since 
$$
\psi(\alpha ;x)
   =
      {_{2}F_1}\left(\begin{matrix} \frac{\alpha }{2},\frac{\alpha +1}{2} \\ \alpha +1\end{matrix};4x\right), 
$$
the function $\psi(\alpha ;x)$ is analytically continued to $4x \in \mathbb{C}\setminus \{0,1\}$ by analytic continuation of ${_{2}F_1}(z)$. 
\begin{lem}
\label{thm:psi lemma}
{\rm{(1)}} Another expression
\begin{align}
\label{eq:series expression}
\psi (\alpha ;x)
   =
   \sum_{k=0}^{\infty}
   c_{\alpha +2k-1,k}x^{k}.
\end{align}
{\rm{(2)}} Closed form
\begin{align}
\label{eq:explicit expression}
\psi (\alpha ;x)
   =
   \left(\frac{1-\sqrt{1-4x}}{2x}\right)^{\alpha }.
\end{align}
{\rm{(3)}} Index law
\begin{align}
\label{eq:index law}
\psi(\alpha ;x)\psi(\beta ;x)=\psi(\alpha +\beta ;x).
\end{align}
{\rm{(4)}} Quadratic formula
\begin{align}
\label{eq:solutions of quadratic eq}
x(x\psi(1 ;x^{2}))^{2}-(x\psi(1 ;x^{2}))+x=0. 
\end{align}
\end{lem}
\begin{proof}
{\rm{(1)}} By the definition of $\psi (\alpha ;x)$ and $c_{a,k}$, we have 
\begin{align}
\psi (\alpha ;x)
   &=
   \sum_{k=0}^{\infty}
      \frac{(\alpha )_{2k}}{k!(\alpha +1)_{k}}x^{k} \nonumber \\
   &=
   1
   +\sum_{k=1}^{\infty}
   (\alpha +k-k)\frac{(\alpha + k + 1)_{k-1}}{k!}x^{k} \nonumber \\
   &=
   1
   +\sum_{k=1}^{\infty}
   \left(\frac{(\alpha + k )_{k}}{k!}-\frac{(\alpha + k + 1)_{k-1}}{(k-1)!}\right)x^{k} \nonumber \\
   &=
   1
   +\sum_{k=1}^{\infty}
    \left(\binom{\alpha +2k-1}{k}-\binom{\alpha +2k-1}{k-1}\right)x^{k} \nonumber \\
   &=
    \sum_{k=0}^{\infty}c_{\alpha +2k-1,k}x^{k}. \nonumber
\end{align}
{\rm{(2)}} We remark a hypergeometric transformation \cite{AAR} (3.1.10)
\begin{equation}
{_{2}F_1}\left(\begin{matrix} \frac{\alpha }{2},\frac{\alpha +1}{2} \\ \alpha -\beta +1\end{matrix};x\right)
   =
   \left(2\frac{1-\sqrt{1-x}}{x}\right)^{\alpha }
   {_{2}F_1}\left(\begin{matrix} \alpha ,\beta  \\ \alpha -\beta +1\end{matrix};\frac{1-\sqrt{1-x}}{1+\sqrt{1-x}}\right). \nonumber 
\end{equation}
Thus, 
$$
\psi(\alpha ;x)
   =
   {_{2}F_1}\left(\begin{matrix} \frac{\alpha }{2},\frac{\alpha +1}{2} \\ \alpha +1\end{matrix};4x\right)
   =
   \left(\frac{1-\sqrt{1-4x}}{2x}\right)^{\alpha }.
$$
The formulas (\ref{eq:index law}) and (\ref{eq:solutions of quadratic eq}) follow from (\ref{eq:explicit expression}) immediately. 
\end{proof}
The following Lemma is a corollary of Lemma \ref{thm:psi lemma} and the key step in the proof of Theorem \ref{thm:the first kind}.
\begin{lem}
\label{thm:key lemma}
If 
$$
y:=\frac{1-\sqrt{1-4x^{2}}}{2x}, 
$$
then 
$$
x=\frac{y}{1+y^{2}}
$$
and for any non-negative integer $N$ 
\begin{align}
\label{eq:key formula}
y^{N}
   =
   x^{N}\psi\left(N;x^{2}\right)
   =
   x^{N}\sum_{k=0}^{\infty}c_{N +2k-1,k}x^{2k}. 
\end{align}
\end{lem}
\begin{proof}
By quadratic formula, 
$$
y=x\frac{1-\sqrt{1-4x^{2}}}{2x^{2}}=x\psi\left(1;x^{2}\right).
$$
Thus, we have 
\begin{align}
y^{N}
   =
   x^{N}\psi\left(1;x^{2}\right)^{N}
   =
   x^{N}\psi\left(N;x^{2}\right)
   =
   x^{N}\sum_{k=0}^{\infty}c_{N +2k-1,k}x^{2k}. \nonumber 
\end{align}
Here the second and third equalities follow from (\ref{eq:index law}) and (\ref{eq:series expression}) respectively. 
\end{proof}
\begin{rem}
We mention some properties of $c_{n,k}$. 
For non-negative integers $n$ and $k$, we make the table of $c_{n,k}>0$. 
\begin{table}[h]
\begin{center}
  \begin{tabular}{|c||c|c|c|c|c|c|c|c|c|c|c|c|} \hline
   $n\setminus k$ & $0$ & $1$ & $2$ & $3$ & $4$ & $5$ & $6$ & $7$ & $8$ & $9$ & $10$  \\ \hline \hline
   $0$ & $1$ &   &  &  &  &  &  &  &  &  & \\ \hline
   $1$ & $1$ &   &  &  &  &  &  &  &  &  & \\ \hline
   $2$ & $1$ & $1$ &  &  &  &  &  &  &  &  & \\ \hline
   $3$ & $1$ & $2$ &  &  &  &  &  &  &  &  & \\ \hline
   $4$ & $1$ & $3$ & $2$ &  &  &  &  &  &  &  & \\ \hline
   $5$ & $1$ & $4$ & $5$ &  &  &  &  &  &  &  & \\ \hline
   $6$ & $1$ & $5$ & $9$ & $5$ &  &  &  &  &  &   & \\ \hline
   $7$ & $1$ & $6$ & $14$ & $14$ &  &  &  &  &  &  & \\ \hline
   $8$ & $1$ & $7$ & $20$ & $28$ & $14$ &  &  &  &  &  & \\ \hline
   $9$ & $1$ & $8$ & $27$ & $48$ & $42$ &  &  &  &  &  & \\ \hline
   $10$ & $1$ & $9$ & $35$ & $75$ & $90$ & $42$ &  &  &  &  & \\ \hline
   $11$ & $1$ & $10$ & $44$ & $110$ & $165$ & $132$ &  &  &  &  & \\ \hline
   $12$ & $1$ & $11$ & $54$ & $154$ & $275$ & $297$ & $132$ &  &  &  &  \\ \hline
   $13$ & $1$ & $12$ & $65$ & $208$ & $429$ & $572$ & $429$ &  &  &  &  \\ \hline
   $14$ & $1$ & $13$ & $77$ & $273$ & $637$ & $1001$ & $1001$ & $429$ &  &  &  \\ \hline
   $15$ & $1$ & $14$ & $90$ & $350$ & $910$ & $1638$ & $2002$ & $1430$ &  &  &  \\ \hline
   $16$ & $1$ & $15$ & $104$ & $440$ & $1260$ & $2548$ & $3640$ & $3432$ & $1430$ &  &  \\ \hline 
   $17$ & $1$ & $16$ & $119$ & $544$ & $1700$ & $3808$ & $6188$ & $7072$ & $4862$ &  &  \\ \hline 
   $18$ & $1$ & $17$ & $135$ & $663$ & $2244$ & $5508$ & $9996$ & $13260$ & $11934$ & $4862$ & \\ \hline 
   $19$ & $1$ & $18$ & $152$ & $798$ & $2907$ & $7752$ & $15504$ & $23256$ & $25194$ & $16796$ & \\ \hline 
   $20$ & $1$ & $19$ & $170$ & $950$ & $3705$ & $10659$ & $23256$ & $38760$ & $48450$ & $41990$ & $16796$ \\ \hline 
   $21$ & $1$ & $20$ & $189$ & $1120$ & $4655$ & $14364$ & $33915$ & $62016$ & $87210$ & $90440$ & $58786$ \\ \hline 
  \end{tabular}
\caption{$c_{n,k}$}
\end{center}
\end{table}
This table is determined exactly by initial conditions 
\begin{align}
c_{n,0}=&1 \quad (n\geq 0), \nonumber \\
c_{n,k}:=&0 \quad \left(n>0, \left\lfloor \frac{n}{2}\right\rfloor <k\right), \nonumber \\
c_{2n,n+1}=&\frac{1}{n+2}\binom{2n+2}{n+1} \quad (n\geq 0), \nonumber \\
c_{2n,n}=&\frac{1}{n+1}\binom{2n}{n}: \text{Catalan numbers} \quad (n\geq 0) \nonumber 
\end{align}
and a recursion formula
$$
c_{n,k}=c_{n-1,k-1}+c_{n-1,k}.
$$
The sequence $c_{n,k}$ is a kind of Clebsch-Gordan coefficients for the Lie algebra $sl_{2}$. 
In fact, from the above initial conditions and recursion of $c_{n,k}$, we have 
\begin{equation}
\label{eq:Kostka}
   \left(\frac{\sin{(2\theta )}}{\sin{\theta }}\right)^{n}
   =
   (2\cos{\theta })^{n}
   =
   \sum_{k=0}^{\left\lfloor \frac{n}{2}\right\rfloor}
   c_{n,k}\frac{\sin{((n-2k+1)\theta )}}{\sin{\theta }}, 
\end{equation}
that is the classical Clebsch-Gordan rule for $sl_{2}$ exactly.

Further $c_{n,k}$ is also a typical example of Kostka numbers $K_{\lambda \mu }$ (see \cite{S} Chapter\,2 Section\,2.11 and Chapter\,4 Section\,4.9). 
We remark Young's rule
\begin{align}
\label{eq:Young rule}
s_{(\mu _{1})}(\mathbf{z})\cdots s_{(\mu _{n})}(\mathbf{z})
   =\sum_{\lambda }K_{\lambda \mu }s_{\lambda }(\mathbf{z}) \quad (\mu _{1}\geq \cdots \geq \mu _{n}\geq 1),
\end{align}
and 
\begin{align}
s_{(\lambda _{1},\lambda _{2})}\left(e^{\sqrt{-1\theta }}, e^{-\sqrt{-1}\theta }\right)
   =
   \frac{e^{(\lambda _{1}-\lambda _{2}+1)\sqrt{-1\theta }}-e^{-(\lambda _{1}-\lambda _{2}+1)\sqrt{-1\theta }}}{e^{\sqrt{-1\theta }}-e^{-\sqrt{-1\theta }}}
   =
   \frac{\sin{((\lambda _{1}-\lambda _{2}+1)\theta )}}{\sin{\theta }}. \nonumber 
\end{align}
By putting $r=2$, $\mu _{1}=\cdots =\mu _{n}=1, z_{1}=e^{\sqrt{-1}\theta }, z_{2}=e^{-\sqrt{-1}\theta }$ in (\ref{eq:Young rule}), we have 
\begin{align}
\left(\frac{\sin{(2\theta )}}{\sin{\theta }}\right)^{n}
   &=
   s_{(1,0)}\left(e^{\sqrt{-1\theta }}, e^{-\sqrt{-1}\theta }\right)^{n} \nonumber \\
   &=
   \sum_{\lambda }K_{\lambda (1^{n}) }s_{\lambda }\left(e^{\sqrt{-1\theta }}, e^{-\sqrt{-1}\theta }\right) \nonumber \\
\label{eq:Kostka2}
   &=
   \sum_{\lambda }K_{\lambda (1^{n}) }\frac{\sin{((\lambda _{1}-\lambda _{2}+1)\theta )}}{\sin{\theta }}.
\end{align}
Finally, by comparing (\ref{eq:Kostka}) and (\ref{eq:Kostka2}), we have 
$$
K_{(n-k,k) ,(1^{n})}
   =
   c_{n,k}.
$$
\end{rem}

\section{Proofs of Theorem \ref{thm:the first kind} and Theorem \ref{thm:the second kind}}
From (\ref{eq:gen fnc of elemental}), (\ref{eq:gen fnc of complete}) and simple calculation
$$
(1\pm z_{j}y)(1\pm z_{j}^{-1}y)
   =1\pm (z_{j}+z_{j}^{-1})y+y^{2}
   =(1+y^{2})\left(1\pm (z_{j}+z_{j}^{-1})\frac{y}{1+y^{2}}\right),
$$
we obtain the following key lemma. 
\begin{lem}
\label{thm:gen fnc of CF}
If 
$$
|z_{j}y|, |z_{j}^{-1}y|, \left|(z_{j}+z_{j}^{-1})\frac{y}{1+y^{2}}\right|<1 \quad (j=1,\ldots ,r), 
$$
then
\begin{align}
\sum_{n=0}^{2r}
   e_{n}^{(2r)}(\mathbf{z},\mathbf{z}^{-1})y^{n}
   &=
   \prod_{j=1}^{r}(1+z_{j}y)(1+z_{j}^{-1}y) \nonumber \\
\label{eq:gen fnc of T and C}
   &=
   (1+y^{2})^{r}
   \sum_{m=0}^{r}
   e_{m}^{(r)}(\mathbf{z}+\mathbf{z}^{-1})\left(\frac{y}{1+y^{2}}\right)^{m}, \\
\sum_{n=0}^{\infty}
   h_{n}^{(2r)}(\mathbf{z},\mathbf{z}^{-1})y^{n}
   &=
   \prod_{j=1}^{r}\frac{1}{(1-z_{j}y)(1-z_{j}^{-1}y)} \nonumber \\
\label{eq:gen fnc of S and F}
   &=
   (1+y^{2})^{-r}
   \sum_{m=0}^{\infty}
   h_{m}^{(r)}(\mathbf{z}+\mathbf{z}^{-1})\left(\frac{y}{1+y^{2}}\right)^{m}.
\end{align}
\end{lem}
By Lemma \ref{thm:gen fnc of CF} and the definition of the power symmetric polynomials, we prove Theorem \ref{thm:the first kind} and Theorem \ref{thm:the second kind}. \\

\noindent
{\bf{Proof of Theorem \ref{thm:the first kind}}} 
{\rm{(1)}}  
By (\ref{eq:gen fnc of T and C}),  
\begin{align}
\label{eq:cal step}
\sum_{n=0}^{2r}
   e_{n}^{(2r)}(\mathbf{z},\mathbf{z}^{-1})y^{n-r}
   =
   x^{-r}\sum_{m=0}^{r}
   e_{m}^{(r)}(\mathbf{z}+\mathbf{z}^{-1})x^{m},
\end{align}
where 
$$
x=\frac{y}{1+y^{2}}. 
$$
Hence, from (\ref{eq:key formula}) we have 
\begin{align}
\sum_{n=0}^{2r}
   e_{n}^{(2r)}(\mathbf{z},\mathbf{z}^{-1})y^{n-r}
   &=
   \sum_{n=0}^{2r}
   e_{n}^{(2r)}(\mathbf{z},\mathbf{z}^{-1})x^{n-r}\sum_{k=0}^{\infty}c_{n-r +2k-1,k}x^{2k} \nonumber \\
\label{eq:cal step2}
   &=
   x^{-r}\sum_{m=0}^{\infty}
   \sum_{k=\max\left(\left\lfloor \frac{m}{2}\right\rfloor -r,0\right)}^{\left\lfloor \frac{m}{2}\right\rfloor}
   e_{m-2k}^{(2r)}(\mathbf{z},\mathbf{z}^{-1})
   c_{m-r -1,k}x^{m}. 
\end{align}
By comparing coefficients of (\ref{eq:cal step}) and (\ref{eq:cal step2}), we obtain the conclusion. \\
{\rm{(2)}} Similarly, from (\ref{eq:gen fnc of S and F})  
\begin{align}
\label{eq:cal step3}
\sum_{n=0}^{\infty}
   h_{n}^{(2r)}(\mathbf{z},\mathbf{z}^{-1})y^{n+r}
   =
   x^{r}\sum_{m=0}^{\infty}
   h_{m}^{(r)}(\mathbf{z}+\mathbf{z}^{-1})x^{m}
\end{align}
and (\ref{eq:key formula}) we have
\begin{align}
\sum_{n=0}^{\infty}
   h_{n}^{(2r)}(\mathbf{z},\mathbf{z}^{-1})y^{n+r}
   &=
\sum_{n=0}^{\infty}
   h_{n}^{(2r)}(\mathbf{z},\mathbf{z}^{-1})x^{n+r}\sum_{k=0}^{\infty}c_{n+r +2k-1,k}x^{2k} \nonumber \\
   &=
\label{eq:cal step4}   
x^{r}\sum_{m=0}^{\infty}
   \sum_{k=0}^{\left\lfloor \frac{n}{2}\right\rfloor}
   h_{m-2k}^{(2r)}(\mathbf{z},\mathbf{z}^{-1})c_{m+r -1,k}x^{m}. 
\end{align}
The formula (\ref{eq:main complete}) follows from (\ref{eq:cal step3}) and (\ref{eq:cal step4}).\\
{\rm{(3)}} We prove this formula without generating function and other Lemmas. 
In fact, by applying the usual binomial formula we have 
\begin{align}
p_{m}^{(r)}(\mathbf{z}+\mathbf{z}^{-1})
   &=
   \sum_{i=1}^{r}(z_{i}+z_{i}^{-1})^{m} \nonumber \\
   &=
   \sum_{i=1}^{r}
   \sum_{k=0}^{m}
   \binom{m}{k}
   z_{i}^{(2k-m)} \nonumber \\
   &=
   \frac{1}{2}
   \sum_{i=1}^{r}
   \sum_{k=0}^{m}
   \left(\binom{m}{k}z_{i}^{2k-m}
   +\binom{m}{m-k}z_{i}^{m-2k}\right)
 \nonumber \\
   &=
   \frac{1}{2}
   \sum_{k=0}^{m}
   \binom{m}{k}
   \sum_{i=1}^{r}
   (z_{i}^{2k-n}+z_{i}^{n-2k}) \nonumber \\
   &=
   \frac{1}{2}
   \sum_{k=0}^{m}
   \binom{m}{k}
   p_{|m-2k|}^{(2r)}(\mathbf{z},\mathbf{z}^{-1}). \nonumber
\end{align}
\qed

\noindent
{\bf{Proof of Theorem \ref{thm:the second kind}}} 
{\rm{(1)}} 
From (\ref{eq:gen fnc of T and C}) and binomial theorem, we have
\begin{align}
\sum_{n=0}^{2r}
   e_{n}^{(2r)}(\mathbf{z},\mathbf{z}^{-1})y^{n}
   &=
   \sum_{m=0}^{r}
   e_{m}^{(r)}(\mathbf{z}+\mathbf{z}^{-1})y^{m}(1+y^{2})^{r-m} \nonumber \\
   &=
   \sum_{m=0}^{r}
   \sum_{k=0}^{r-m}
   \binom{r-m}{k}
   e_{m}^{(r)}(\mathbf{z}+\mathbf{z}^{-1})y^{m+2k} \nonumber \\
   &=
   \sum_{n=0}^{2r}
   \sum_{k=\max\left\{\left\lfloor \frac{n-r}{2}\right\rfloor,0\right\}}^{\left\lfloor \frac{n}{2}\right\rfloor}
   \binom{r-n+2k}{k}
   e_{n-2k}^{(r)}(\mathbf{z}+\mathbf{z}^{-1})y^{n}. \nonumber
\end{align}
{\rm{(2)}} 
Similarly, we have 
\begin{align}
\sum_{n=0}^{\infty}
   h_{n}^{(2r)}(\mathbf{z},\mathbf{z}^{-1})y^{n}
   &=
   \sum_{m=0}^{\infty}
   h_{m}^{(r)}(\mathbf{z}+\mathbf{z}^{-1})y^{m}(1+y^{2})^{-m-r} \nonumber \\
   &=
   \sum_{m=0}^{\infty}
   \sum_{k=0}^{\infty}
   \frac{(m+r)_{k}}{k!}
   h_{m}^{(r)}(\mathbf{z}+\mathbf{z}^{-1})
   y^{m+2k} \nonumber \\
   &=
   \sum_{n=0}^{\infty}
   \sum_{k=0}^{\left\lfloor \frac{n}{2}\right\rfloor}
   \frac{(n-2k+r)_{k}}{k!}
   h_{n-2k}^{(r)}(\mathbf{z}+\mathbf{z}^{-1})
   y^{n} \nonumber \\
   &=
   \sum_{n=0}^{\infty}
   \sum_{k=0}^{\left\lfloor \frac{n}{2}\right\rfloor}
   \binom{n-k+r-1}{k}
   h_{n-2k}^{(r)}(\mathbf{z}+\mathbf{z}^{-1})
   y^{n}. \nonumber
\end{align}
(3) It is enough to show that the case of $r=1$ which is 
$$
z^{n}+z^{-n}
   =
   2\sum_{k= 0}^{\lfloor \frac{n+1}{2}\rfloor}\binom{2k-n-1}{k}(z+z^{-1})^{n-2k}
   -\sum_{k= 0}^{\lfloor \frac{n}{2}\rfloor}\binom{2k-n}{k}(z+z^{-1})^{n-2k}.
$$
From (\ref{eq:gen fnc of T and C}) and simple calculus, 
\begin{align}
\sum_{n= 0}^{\infty}(z^{n}+z^{-n})y^{n}
   &=
   \frac{1}{1-zy}+\frac{1}{1-z^{-1}y} \nonumber \\
   &=
   \frac{2-(z+z^{-1})y}{1-(z+z^{-1})y+y^{2}} \nonumber \\
   &=
   \frac{1}{1+y^{2}}\frac{2-(z+z^{-1})y}{1-(z+z^{-1})\frac{y}{1+y^{2}}}. \nonumber
\end{align}
If $|zy|<1$, $|z^{-1}y|<1$ and $|y|<1$, then
\begin{align}
\sum_{n= 0}^{\infty}(z^{n}+z^{-n})y^{n}
   &=
   \sum_{m= 0}^{\infty}\left(\frac{2}{y}-(z+z^{-1})\right)(z+z^{-1})^{m}\left(\frac{y}{1+y^{2}}\right)^{m+1} \nonumber \\
   &=
   \sum_{m= 0}^{\infty}\sum_{k= 0}^{\infty}
   \left(\frac{2}{y}-(z+z^{-1})\right)(z+z^{-1})^{m}y^{m+1}\frac{(m+1)_{k}}{k!}(-1)^{k}y^{2k} \nonumber \\
   &=
   \sum_{n= 0}^{\infty}2\sum_{k= 0}^{\lfloor \frac{n+1}{2}\rfloor}\binom{2k-n-1}{k}(z+z^{-1})^{n-2k}y^{n} \nonumber \\
   & \quad 
   -\sum_{n= 0}^{\infty}\sum_{k= 0}^{\lfloor \frac{n+1}{2}\rfloor}\binom{2k-n}{k}(z+z^{-1})^{n-2k+1}y^{n+1}. \nonumber 
\end{align} \qed

Finally we consider the principal specialization of Theorem \ref{thm:the first kind} and Theorem \ref{thm:the second kind}, which means substituting
$$
\boldsymbol{q}^{\pm \iota }:=(q^{\pm r},\ldots ,q^{\pm 1})
$$
for $\mathbf{z}$. 
In this special case, we evaluate $e_{n}^{(2r)}(\boldsymbol{q}^{+\iota },\boldsymbol{q}^{-\iota })$, $h_{n}^{(2r)}(\boldsymbol{q}^{+\iota },\boldsymbol{q}^{-\iota })$ and $p_{n}^{(2r)}(\boldsymbol{q}^{+\iota },\boldsymbol{q}^{-\iota })$ explicitly.  
\begin{prop}
\label{thm:prop1}
For any non-negative integer $n$, we have the following identities. \\
{\rm{(1)}}  
\begin{align}
   \sum_{k=0}^{\min(n,2r+1)}
   (-1)^{n-k}\binom{2r+1}{k}_{q}q^{\frac{k(k-2r-1)}{2}}
   =
   \begin{cases}
   e_{n}^{(2r)}(\boldsymbol{q}^{+\iota },\boldsymbol{q}^{-\iota }) & (n=0,1,\ldots, 2r) \\
   0 & ({\rm{others}})
   \end{cases}.
\end{align}
{\rm{(2)}}  
\begin{align}
h_{n}^{(2r)}(\boldsymbol{q}^{+\iota },\boldsymbol{q}^{-\iota })
   =
q^{-nr}\left(\binom{2r+n}{n}_{q}
   -\binom{2r+n-1}{n-1}_{q}q^{r}\right).
\end{align}
{\rm{(3)}} 
\begin{align}
p_{n}^{(2r)}(\boldsymbol{q}^{+\iota },\boldsymbol{q}^{-\iota })
   =
   -1+q^{-rn}\frac{1-q^{(2r+1)n}}{1-q^{n}}. 
\end{align}
\end{prop}
\begin{proof}
{\rm{(1)}} From generating function of elementary symmetric polynomials (\ref{eq:gen fnc of elemental}), 
\begin{align}
\sum_{n=0}^{2r}
   e_{n}^{(2r)}(\boldsymbol{q}^{+\iota },\boldsymbol{q}^{-\iota })y^{n}
   &=
   \prod_{j=1}^{r}(1+q^{j}y)(1+q^{-j}y)
   =
   \frac{1}{1+y}\prod_{j=0}^{2r}(1+q^{j}q^{-r}y). \nonumber 
\end{align}
By $q$-binomial formula (\ref{eq:q-binom}), we have 
\begin{align}
\sum_{n=0}^{2r}
   e_{n}^{(2r)}(\boldsymbol{q}^{+\iota },\boldsymbol{q}^{-\iota })y^{n}
   &=
   \sum_{i=0}^{\infty}(-1)^{i}y^{i}
   \sum_{k=0}^{2r+1}\binom{2r+1}{k}_{q}q^{\frac{k(k-1)}{2}}q^{-kr}y^{k} \nonumber \\
   &=
   \sum_{n=0}^{\infty}
   \sum_{k=0}^{\min(n,2r+1)}
   (-1)^{n-k}
   \binom{2r+1}{k}_{q}q^{\frac{k(k-2r-1)}{2}}y^{n}.
\nonumber
\end{align}
{\rm{(2)}} Similarly, we have 
\begin{align}
\sum_{n=0}^{\infty}
   h_{n}^{(2r)}(\boldsymbol{q}^{+\iota },\boldsymbol{q}^{-\iota })y^{n}
   &=
   \prod_{j=1}^{r}\frac{1}{(1-q^{j}y)(1-q^{-j}y)} \nonumber \\
   &=
   (1-y)\prod_{j=0}^{2r}\frac{1}{1-q^{j}q^{-r}y} \nonumber \\
   &=
   (1-y)\sum_{k=0}^{\infty}\binom{2r+k}{k}_{q}q^{-kr}y^{k} \nonumber \\
   &=
   \sum_{n=0}^{\infty}
   \left(\binom{2r+n}{n}_{q}
   -\binom{2r+n-1}{n-1}_{q}q^{r}\right)q^{-nr}y^{n}. \nonumber
\end{align}
{\rm{(3)}} By the definition of power sum and geometric series, we have 
\begin{align}
p_{n}^{(2r)}(\boldsymbol{q}^{+\iota },\boldsymbol{q}^{-\iota })
   =
   -1+\sum_{k=-r}^{r}q^{kn}
   =
   -1+\frac{q^{-rn}-q^{(r+1)n}}{1-q^{n}}. \nonumber
\end{align}
We remark this formula holds on the limit $q\rightarrow 1$. 
\end{proof}
\begin{cor}
{\rm{(1)}} For any non-negative integer $m$, 
\begin{align}
& \sum_{k=\max\left\{\left\lfloor \frac{m-r}{2}\right\rfloor,0\right\}}^{\left\lfloor \frac{m}{2}\right\rfloor}
   \!\!\!\!\!\!\!\!\!\!\!\!\!\!
   \sum_{l=0}^{\min(m-2k,2r+1)}
   \!\!\!\!\!\!\!\!\!\!\!\!\!\!
   (-1)^{m-l}
   \binom{2r+1}{l}_{q}q^{\frac{l(l-2r-1)}{2}}
   c_{m-r -1,k} \nonumber \\
   & \quad \quad =
   \begin{cases}
   e_{m}^{(r)}(\boldsymbol{q}^{+\iota }+\boldsymbol{q}^{-\iota }) & (m=0,1,\ldots, r) \\
   0 & ({\rm{others}})
   \end{cases}.
\end{align}
For $n=0,1,\ldots, 2r$, 
\begin{align}
\sum_{k=0}^{n}
   (-1)^{k}
   \binom{2r+1}{k}_{q}q^{\frac{k(k-2r-1)}{2}}
   =
   \sum_{k=\max\left\{\left\lfloor \frac{n-r}{2}\right\rfloor,0\right\}}^{\left\lfloor \frac{n}{2}\right\rfloor}
   \binom{r-n+2k}{k}
   e_{n-2k}^{(r)}(\boldsymbol{q}^{+\iota }+\boldsymbol{q}^{-\iota }).
\end{align}
{\rm{(2)}} For any non-negative integers $n$ and $m$, we have
\begin{align}
& h_{m}^{(r)}(\boldsymbol{q}^{+\iota }+\boldsymbol{q}^{-\iota }) \nonumber \\
   & \quad =
   \sum_{k= 0}^{\left\lfloor \frac{m}{2}\right\rfloor}
   (-1)^{m}q^{-(m-2k)r}\left(\binom{2r+m-2k}{m-2k}_{q}
   -\binom{2r+m-2k-1}{m-2k-1}_{q}q^{r}\right)
   c_{m+r-1,k}, \\ 
& q^{-nr}\left(\binom{2r+n}{n}_{q}
   -\binom{2r+n-1}{n-1}_{q}q^{r}\right)
   =
   \sum_{k=0}^{\left\lfloor \frac{n}{2}\right\rfloor}
   \binom{n-k+r-1}{k}
   h_{n-2k}^{(r)}(\boldsymbol{q}^{+\iota }+\boldsymbol{q}^{-\iota }). 
\end{align}
{\rm{(3)}} For any non-negative integer $m$, 
\begin{align}
p_{m}^{(r)}(\boldsymbol{q}^{+\iota }+\boldsymbol{q}^{-\iota })
   &=
   -2^{m-1}
   +\frac{1}{2}\sum_{k=0}^{m}\binom{m}{k}\frac{1-q^{(2r+1)|m-2k|}}{1-q^{|m-2k|}}q^{-|m-2k|r}. 
\end{align}
For any positive integer $n$, 
\begin{align}
& -1+\frac{q^{-rn}-q^{(r+1)n}}{1-q^{n}} \nonumber \\
& \quad =
   2\sum_{k= 0}^{\lfloor \frac{n+1}{2}\rfloor}\binom{2k-n-1}{k}
   p_{n-2k}^{(r)}(\boldsymbol{q}^{+\iota }+\boldsymbol{q}^{-\iota })
   -\sum_{k= 0}^{\lfloor \frac{n}{2}\rfloor}\binom{2l-n}{k}p_{n-2k}^{(r)}(\boldsymbol{q}^{+\iota }+\boldsymbol{q}^{-\iota }).
\end{align}
\end{cor}

\section{Applications to $F_{n}^{(r)}$ and $L_{n}^{(r)}$}
In this section we investigate more specializations of Theorem \ref{thm:the first kind} and Theorem \ref{thm:the second kind}, and prove Theorem \ref{thm:explicit formulas}, Theorem \ref{thm:inversion of explicit formulas} and Theorem \ref{thm:initial values and recurrence formulas}. 
To apply Theorem \ref{thm:the first kind} and Theorem \ref{thm:the second kind} to $F_{n}^{(r)}$ and $L_{n}^{(r)}$, we evaluate  $e_{n}^{(2r)}(-\boldsymbol{\zeta}^{+\iota },-\boldsymbol{\zeta}^{-\iota })$, $h_{n}^{(2r)}(-\boldsymbol{\zeta}^{+\iota },-\boldsymbol{\zeta}^{-\iota })$, $p_{n}^{(2r)}(-\boldsymbol{\zeta}^{+\iota },-\boldsymbol{\zeta}^{-\iota })$ and $e_{m}^{(r)}(-\boldsymbol{\zeta}^{+\iota }-\boldsymbol{\zeta}^{-\iota })$. 
\begin{prop}
\label{thm:prop2}
{\rm{(1)}} For $n=0,1,\ldots, 2r$, we have 
\begin{align}
\label{eq:e2r evaluation}
e_{n}^{(2r)}(-\boldsymbol{\zeta}^{+\iota },-\boldsymbol{\zeta}^{-\iota })
   &=
   1. 
\end{align}
{\rm{(2)}} For any non-negative integer $n$, 
\begin{align}
h_{n}^{(2r)}(-\boldsymbol{\zeta}^{+\iota },-\boldsymbol{\zeta}^{-\iota })
   &=
   \frac{1}{2}\left((-1)^{\left\lfloor \frac{n}{2r+1}\right\rfloor}-(-1)^{\left\lfloor \frac{n-2}{2r+1}\right\rfloor}\right) \nonumber \\
\label{eq:h2r evaluation}
   &=\begin{cases}
   1 & (n\equiv 0,1\,(\mathrm{mod}\,4r+2)) \\
   -1 & (n\equiv 2r+1,2r+2\,(\mathrm{mod}\,4r+2)) \\
   0 & ({\rm{others}})
   \end{cases}.
\end{align}
{\rm{(3)}} For any non-negative integer $n$, 
\begin{align}
\label{eq:p2r evaluation}
p_{n}^{(2r)}(-\boldsymbol{\zeta}^{+\iota },-\boldsymbol{\zeta}^{-\iota })
   &=
   (-1)^{n}(-1+(2r+1)\delta_{2r+1\mid n})
   =\begin{cases}
   (-1)^{n-1} & (2r+1\not| \,\,\,n) \\
   (-1)^{n}2r & (2r+1\mid n)
   \end{cases},
\end{align}
where 
$$
\delta_{2r+1\mid n}
   :=
   \begin{cases}
   0 & (2r+1\not| \,\,\,n) \\
   1 & (2r+1\mid n)
   \end{cases}. 
$$
\end{prop}
\begin{proof}
{\rm{(1)}} From (\ref{eq:gen fnc of elemental}), we have 
\begin{align}
\sum_{n=0}^{2r}
   e_{n}^{(2r)}(-\boldsymbol{\zeta}^{+\iota },-\boldsymbol{\zeta}^{-\iota })y^{n}
   &=
   \prod_{j=1}^{r}\left(1-e^{2\pi \sqrt{-1}\frac{j}{2r+1}}y\right)\left(1-e^{-2\pi \sqrt{-1}\frac{j}{2r+1}}y\right) \nonumber \\
   &=
   \frac{1-y^{2r+1}}{1-y} \nonumber \\
   &=
   \sum_{n=0}^{2r}y^{n}. \nonumber 
\end{align}
{\rm{(2)}} From (\ref{eq:gen fnc of complete}), we have 
\begin{align}
\sum_{n=0}^{\infty}
   h_{n}^{(2r)}(-\boldsymbol{\zeta}^{+\iota },-\boldsymbol{\zeta}^{-\iota })y^{n}
   &=
   \prod_{j=1}^{r}\frac{1}{\left(1+e^{2\pi \sqrt{-1}\frac{j}{2r+1}}y\right)\left(1+e^{-2\pi \sqrt{-1}\frac{j}{2r+1}}y\right)} \nonumber \\
   &=
   \frac{1+y}{1+y^{2r+1}} \nonumber \\
   &=
   \sum_{k=0}^{\infty}(-1)^{k}(y^{(2r+1)k}+y^{(2r+1)k+1}) \nonumber \\
   &=
   \sum_{n=0}^{\infty}
   \frac{1}{2}\left((-1)^{\left\lfloor \frac{n}{2r+1}\right\rfloor}-(-1)^{\left\lfloor \frac{n-2}{2r+1}\right\rfloor}\right)y^{n}. \nonumber  
\end{align}
{\rm{(3)}} By the definition of $p_{n}^{(2r)}$, 
\begin{align}
p_{n}^{(2r)}(-\boldsymbol{\zeta}^{+\iota },-\boldsymbol{\zeta}^{-\iota })
   &=
   (-1)^{n}\left(-1+\sum_{k=-r}^{r}\zeta_{2r+1}^{kn}\right)
   =
   (-1)^{n}(-1+(2r+1)\delta_{2r+1\mid n}). \nonumber 
\end{align}
\end{proof}
\noindent
{\bf{Proof of Theorem \ref{thm:explicit formulas} and Theorem \ref{thm:inversion of explicit formulas}}} 
From Theorem \ref{thm:the first kind} {\rm{(2)}} and {\rm{(3)}} and Proposition 1 {\rm{(2)}} and {\rm{(3)}}, we derive explicit formulas of $F_{n}^{(r)}$ and $L_{n}^{(r)}$. 
Similarly, Theorem \ref{thm:inversion of explicit formulas} follows from Theorem \ref{thm:the first kind} {\rm{(2)}} and {\rm{(3)}} and Proposition 1 {\rm{(2)}} and {\rm{(3)}}. \qed

By specialization of Theorem 3 {\rm{(2)}} and {\rm{(3)}}, we obtain the initial values of $F_{n+1}^{(r)}$ and $L_{n}^{(r)}$.
\begin{cor}
\label{thm:initial conditions}
{\rm{(1)}} If $m\leq r$, then we have
\begin{equation}
F_{2m-1}^{(r+1)}=F_{2m}^{(r)}=\binom{2m+r-2}{m-1}-\binom{2m+r-2}{m-2}.  
\end{equation}
{\rm{(2)}} If $m<2r+1$, then we have 
\begin{align}
L_{2m}^{(r)}
   &=
   -2^{2m-1}
   +\frac{2r+1}{2}\binom{2m}{m}. 
\end{align}
If $m<r$, then we have 
\begin{align}
L_{2m+1}^{(r)}
   &=
   4^{m}.
\end{align}
\end{cor}

To prove Theorem \ref{thm:initial values and recurrence formulas}, we need to evaluate $e_{m}^{(r)}(-\boldsymbol{\zeta}^{+\iota }-\boldsymbol{\zeta}^{-\iota })$, which can be computed from (\ref{eq:main elemental}) and (\ref{eq:e2r evaluation}). 
\begin{prop}
For $m=0,1,\ldots, r$, we have 
\begin{align}
\label{eq:3rd result2}
e_{m}^{(r)}(-\boldsymbol{\zeta}^{+\iota }-\boldsymbol{\zeta}^{-\iota })
   =
   \sum_{k=0}^{\left\lfloor \frac{m}{2}\right\rfloor}
   c_{m-r -1,k}
   =
   \binom{m-r -1}{\left\lfloor \frac{m}{2}\right\rfloor}
   =
   (-1)^{\left\lfloor \frac{m}{2}\right\rfloor}
   \binom{r -\left\lfloor \frac{m+1}{2}\right\rfloor}{\left\lfloor \frac{m}{2}\right\rfloor}.
\end{align}
\end{prop}
\begin{rem}
{\rm{(1)}} 
For (\ref{eq:3rd result2}), we give another proof without using (\ref{eq:main elemental}). 
Let $x:=z+z^{-1}$.
First, we remark
\begin{align}
\sum_{k=-r}^{r}z^{k}
   =
   \prod_{j=1}^{r}\left(x-2\cos\left(\frac{2\pi j}{2r+1}\right)\right)
   =
   \sum_{m=0}^{r}e_{m}^{(r)}(-\boldsymbol{\zeta}^{+\iota }-\boldsymbol{\zeta}^{-\iota })x^{r-m}. \nonumber
\end{align}
On the other hand, if $|u|<|z|<|u|^{-1}$, then 
\begin{align}
\sum_{r=0}^{\infty}u^{r}\sum_{k=-r}^{r}z^{k}
   &=
   \sum_{r=0}^{\infty}u^{r}\left(\frac{z^{-r}}{1-z}-\frac{z^{r+1}}{1-z}\right) \nonumber \\
   &=
   \frac{1}{1-z}
   \left(\frac{1}{1-z^{-1}u}-\frac{z}{1-zu}\right) \nonumber \\
   &=
   \frac{1+u}{(1-z^{-1}u)(1-zu)} \nonumber \\
   &=
   \frac{1+u}{1+u^{2}}\frac{1}{1-x\frac{u}{1+u^{2}}} \nonumber \\
   &=
   (1+u)\sum_{N=0}^{\infty}x^{N}u^{N}(1+u^{2})^{-N-1} \nonumber \\
   &=
   (1+u)\sum_{N=0}^{\infty}x^{N}\sum_{k=0}^{\infty}\binom{-N-1}{k}u^{N+2k} \nonumber \\
   &=
   \sum_{r=0}^{\infty}u^{r}
   \left(
   \sum_{m=0}^{\left\lfloor \frac{r}{2}\right\rfloor }
   (-1)^{m}\binom{r-m}{m}x^{r-2m}
   +\sum_{m=0}^{\left\lfloor \frac{r-1}{2}\right\rfloor }
   (-1)^{m}\binom{r-1-m}{m}x^{r-1-2m}\right). \nonumber
\end{align}
Hence we obtain the conclusion (\ref{eq:3rd result2}).\\
{\rm{(2)}} The formula (\ref{eq:3rd result2}) is obtained by substituting (\ref{eq:e2r evaluation}) into (\ref{eq:main elemental}). 
Similarly, by substituting (\ref{eq:e2r evaluation}) for (\ref{eq:main elemental2}), for $n=0,1,\ldots, 2r$ we obtain 
\begin{equation}
   \sum_{k=\max\left\{\left\lfloor \frac{n-r}{2}\right\rfloor,0\right\}}^{\left\lfloor \frac{n}{2}\right\rfloor}
   \binom{r-n+2k}{k}
   \binom{n-2k-r -1}{\left\lfloor \frac{n}{2}\right\rfloor -k}
   =
   1.
\end{equation}
\end{rem}

\noindent
{\bf{Proof of Theorem \ref{thm:initial values and recurrence formulas}}} 
From the Wronski relations, Newton's formulas and (\ref{eq:3rd result2}), for any non-negative integer $n$ we have 
\begin{align*}
& \sum_{j=0}^{\min{(n,r)}}
   (-1)^{\left\lfloor \frac{j-1}{2}\right\rfloor }
   \binom{r -\left\lfloor \frac{j+1}{2}\right\rfloor}{\left\lfloor \frac{j}{2}\right\rfloor}
   F_{n-j+1}^{(r)}
   =
   0, \\
& \sum_{j=0}^{\min{(n,r)}}(-1)^{n-\left\lfloor \frac{j+1}{2}\right\rfloor}
   \binom{r -\left\lfloor \frac{j+1}{2}\right\rfloor}{\left\lfloor \frac{j}{2}\right\rfloor}L_{n-j+1}^{(r)} \nonumber \\
   & \quad =
   \begin{cases}
   (-1)^{\left\lfloor \frac{n+1}{2}\right\rfloor}(n+1)
   \binom{r -\left\lfloor \frac{n}{2}\right\rfloor -1}{\left\lfloor \frac{n+1}{2}\right\rfloor} & (n=0,1,\ldots ,r-1)\\
   0 & (n\geq r)
   \end{cases}.
\end{align*}
Then $\{F_{n+1}^{(r)}\}_{n\geq 0}$ and $\{L_{n}^{(r)}\}_{n\geq 0}$ satisfy the recursion (\ref{eq:recursion of F and L}).

The initial values of $\{F_{n+1}^{(r)}\}_{n\geq 0}$ are determined by the vanishing property (\ref{eq:vanishing property})
$$
F_{0}^{(r)}=\cdots =F_{-(r-2)}^{(r)}=0
$$
and $F_{1}^{(r)}=h_{0}^{(r)}=1$. 
The initial values of $\{L_{n}^{(r)}\}_{n\geq 0}$ follows from Corollary \ref{thm:initial conditions} {\rm{(2)}}.\qed

\begin{exa}
\noindent
\underline{$r=1$}  
\begin{align}
F_{1}^{(1)}&=1, \,L_{0}^{(1)}=1, \,
a_{n+1}^{(1)}=a_{n}^{(1)}=1. \nonumber
\end{align}
\noindent
\underline{$r=2$ (Fibonacci numbers and Lucas numbers)} 
\begin{align}
F_{-1}^{(2)}&=0, \, F_{0}^{(2)}=1, \, L_{0}^{(2)}=2, \, L_{1}^{(1)}=1, \,
a_{n+2}^{(2)}=a_{n+1}^{(2)}+a_{n}^{(2)}. \nonumber \\
F_{n+1}^{(2)}&: 1,1,2,3,5,8,13,21,34,55,89,144,233,377,610,987,1597,\ldots \nonumber \\
L_{n}^{(2)}&: 2,1,3,4,7,11,18,29,47,76,123,199,322,521,843,1364,\ldots \nonumber 
\end{align}
\underline{$r=3$ (OEIS A006053 and OEIS A096975)} \quad 
\begin{align}
F_{-1}^{(3)}&=F_{0}^{(3)}=0, \, F_{1}^{(3)}=1, \, L_{0}^{(3)}=3, \, L_{1}^{(3)}=1, \, L_{2}^{(3)}=5, \nonumber \\
a_{n+3}^{(3)}&=a_{n+2}^{(3)}+2a_{n+1}^{(3)}-a_{n}^{(3)}. \nonumber \\
F_{n+1}^{(3)}&: 1,1,3,4,9,14,28,47,89,155,286,507,924,1652,2993,\ldots \nonumber \\
L_{n}^{(3)}&: 3,1,5,4,13,16,38,57,117,193,370,639,1186,2094,\ldots \nonumber 
\end{align}
\underline{$r=4$ (OEIS A188021 and OEIS A094649)}
\begin{align}
F_{-2}^{(4)}&=F_{-1}^{(4)}=F_{0}^{(4)}=0, \, F_{1}^{(4)}=1, \nonumber \\
L_{0}^{(4)}&=4, \, L_{1}^{(4)}=1, \, L_{2}^{(4)}=7, \, L_{3}^{(4)}=4, \nonumber \\
a_{n+4}^{(4)}&=a_{n+3}^{(4)}+3a_{n+2}^{(4)}-2a_{n+1}^{(4)}-a_{n}^{(4)}. \nonumber \\
F_{n+1}^{(4)}&: 1,1,4,5,14,20,48,75,165,274,571,988,1988,3536,6953,\ldots \nonumber \\
L_{n}^{(4)}&: 4, 1, 7, 4, 19, 16, 58, 64, 187, 247, 622, 925, 2110, 3394, 7252, \ldots \nonumber 
\end{align}
\underline{$r=5$ (OEIS A231181 and OEIS A189234)} 
\begin{align}
F_{-3}^{(5)}&=F_{-2}^{(5)}=F_{-1}^{(5)}=F_{0}^{(5)}=0, \, F_{1}^{(5)}=1, \nonumber \\ 
L_{0}^{(5)}&=5, \, L_{1}^{(5)}=1, \, L_{2}^{(5)}=9, \, L_{3}^{(5)}=4, \, L_{4}^{(5)}=25, \nonumber \\
a_{n+5}^{(5)}&=a_{n+4}^{(5)}+4a_{n+3}^{(5)}-3a_{n+2}^{(5)}-3a_{n+1}^{(5)}+a_{n}^{(5)}. \nonumber \\
F_{n+1}^{(5)}&: 1,1,5,6,20,27,75,110,275,429,1001,1637,3639,6172,\ldots \nonumber \\
L_{n}^{(5)}&: 5, 1, 9, 4, 25, 16, 78, 64, 257, 256, 874, 1013, 3034, 3953, \ldots \nonumber 
\end{align}
\underline{$r=6$} \quad 
\begin{align}
F_{-4}^{(6)}&=F_{-3}^{(6)}=F_{-2}^{(6)}=F_{-1}^{(6)}=F_{0}^{(6)}=0, \, F_{1}^{(6)}=1, \nonumber \\
L_{0}^{(6)}&=6, \, L_{1}^{(6)}=1, \, L_{2}^{(6)}=11, \, L_{3}^{(6)}=4, \, L_{4}^{(6)}=31, \, L_{5}^{(6)}=16, \nonumber \\
a_{n+6}^{(6)}&=a_{n+5}^{(6)}+5a_{n+4}^{(6)}-4a_{n+3}^{(6)}-6a_{n+2}^{(6)}+3a_{n+1}^{(6)}+a_{n}^{(6)}. \nonumber \\
F_{n+1}^{(6)}&: 1,1,6,7,27,35,110,154,429,637,1638,2548,6188,9995,\ldots \nonumber \\
L_{n}^{(6)}&: 6, 1, 11, 4, 31, 16, 98, 64, 327, 256, 1126, 1024, 3958, 4083, \ldots \nonumber 
\end{align}
\end{exa}

\begin{table}[hbtp]
\begin{center}
  \begin{tabular}{|c||c|c|c|c|c|c|c|c|c|c|c|c|c|c|} \hline
   $r\setminus n$ & 1 & 2 & 3 & 4 & 5 & 6 & 7 & 8 & 9 & 10 & 11 & 12 \\ \hline \hline
   1 & 1 & 1 & 1 & 1 & 1 & 1 & 1 & 1 & 1 & 1 & 1 & 1 \\ \hline
   2 & 1 & 1 & 2 & 3 & 5 & 8 & 13 & 21 & 34 & 55 & 89 & 144\\ \hline
   3 & 1 & 1 & 3 & 4 & 9 & 14 & 28 & 47 & 89 & 155 & 286 & 507 \\ \hline
   4 & 1 & 1 & 4 & 5 & 14 & 20 & 48 & 75 & 165 & 274 & 571 & 988 \\ \hline
   5 & 1 & 1 & 5 & 6 & 20 & 27 & 75 & 110 & 275 & 429 & 1001 & 1637 \\ \hline
   6 & 1 & 1 & 6 & 7 & 27 & 35 & 110 & 154 & 429 & 637 & 1638 & 2548 \\ \hline
   7 & 1 & 1 & 7 & 8 & 35 & 44 & 154 & 208 & 637 & 910 & 2548 & 3808 \\ \hline
   8 & 1 & 1 & 8 & 9 & 44 & 54 & 208 & 273 & 910 & 1260 & 3808 & 5508 \\ \hline
   9 & 1 & 1 & 9 & 10 & 54 & 65 & 273 & 350 & 1260 & 1700 & 5508 & 7752 \\ \hline
   10 & 1 & 1 & 10 & 11 & 65 & 77 & 350 & 440 & 1700 & 2244 & 7752 & 10659 \\ \hline 
   11 & 1 & 1 & 11 & 12 & 77 & 90 & 440 & 544 & 2244 & 2907 & 10659 & 14364 \\ \hline
   12 & 1 & 1 & 12 & 13 & 90 & 104 & 544 & 663 & 2907 & 3705 & 14364 & 19019 \\ \hline
   13 & 1 & 1 & 13 & 14 & 104 & 119 & 663 & 798 & 3705 & 4655 & 19019 & 24794 \\ \hline 
   14 & 1 & 1 & 14 & 15 & 119 & 135 & 798 & 950 & 4655 & 5775 & 24794 & 31878 \\ \hline
   15 & 1 & 1 & 15 & 16 & 135 & 152 & 950 & 1120 & 5775 & 7084 & 31878 & 40480 \\ \hline 
   16 & 1 & 1 & 16 & 17 & 152 & 170 & 1120 & 1309 & 7084 & 8602 & 40480 & 50830 \\ \hline 
  \end{tabular}
  \caption{$F_{n}^{(r)}$}
\end{center}
\begin{center}
  \begin{tabular}{|c||c|c|c|c|c|c|c|c|c|c|c|c|c|c|} \hline
   $r\setminus n$ & 0 & 1 & 2 & 3 & 4 & 5 & 6 & 7 & 8 & 9 & 10 & 11 & 12 & 13 \\ \hline \hline
   1 & 1 & 1 & 1 & 1 & 1 & 1 & 1 & 1 & 1 & 1 & 1 & 1 & 1 & 1 \\ \hline
   2 & 2 & 1 & 3 & 4 & 7 & 11 & 18 & 29 & 47 & 76 & 123 & 199 & 322 & 521 \\ \hline
   3 & 3 & 1 & 5 & 4 & 13 & 16 & 38 & 57 & 117 & 193 & 370 & 639 & 1186 & 2094 \\ \hline
   4 & 4 & 1 & 7 & 4 & 19 & 16 & 58 & 64 & 187 & 247 & 622 & 925 & 2110 & 3394 \\ \hline
   5 & 5 & 1 & 9 & 4 & 25 & 16 & 78 & 64 & 257 & 256 & 874 & 1013 & 3034 & 3953 \\ \hline
   6 & 6 & 1 & 11 & 4 & 31 & 16 & 98 & 64 & 327 & 256 & 1126 & 1024 & 3958 & 4083 \\ \hline
   7 & 7 & 1 & 13 & 4 & 37 & 16 & 118 & 64 & 397 & 256 & 1378 & 1024 & 4882 & 4096 \\ \hline
   8 & 8 & 1 & 15 & 4 & 43 & 16 & 138 & 64 & 467 & 256 & 1630 & 1024 & 5806 & 4096 \\ \hline
   9 & 9 & 1 & 17 & 4 & 49 & 16 & 158 & 64 & 537 & 256 & 1882 & 1024 & 6730 & 4096 \\ \hline
   10 & 10 & 1 & 19 & 4 & 55 & 16 & 178 & 64 & 607 & 256 & 2134 & 1024 & 7654 & 4096 \\ \hline 
   11 & 11 & 1 & 21 & 4 & 61 & 16 & 198 & 64 & 677 & 256 & 2386 & 1024 & 8578 & 4096 \\ \hline
   12 & 12 & 1 & 23 & 4 & 67 & 16 & 218 & 64 & 747 & 256 & 2638 & 1024 & 9502 & 4096 \\ \hline
   13 & 13 & 1 & 25 & 4 & 73 & 16 & 238 & 64 & 817 & 256 & 2890 & 1024 & 10426 & 4096 \\ \hline 
   14 & 14 & 1 & 27 & 4 & 79 & 16 & 258 & 64 & 887 & 256 & 3142 & 1024 & 11350 & 4096 \\ \hline
   15 & 15 & 1 & 29 & 4 & 85 & 16 & 278 & 64 & 957 & 256 & 3394 & 1024 & 11274 & 4096 \\ \hline 
   16 & 16 & 1 & 31 & 4 & 91 & 16 & 298 & 64 & 1027 & 256 & 3646 & 1024 & 13198 & 4096 \\ \hline
  \end{tabular}
  \caption{$L_{n}^{(r)}$}
\end{center}
\end{table}
Finally, we mention some interesting examples of our results, including the seemingly new formulas for the Fibonacci and Lucas numbers. 
\begin{cor}
\label{thm:new Fibonacci formulas}
For any non-negative integers $m$ and $n$, we have 
\begin{align}
   \sum_{k= 0}^{\left\lfloor \frac{m}{2}\right\rfloor}
   \frac{1}{2}\left((-1)^{\left\lfloor \frac{m-2k}{3}\right\rfloor}-(-1)^{\left\lfloor \frac{m-2k-2}{3}\right\rfloor}\right)
   c_{m,k}
   &=
   \sum_{k= 0}^{\left\lfloor \frac{n-1}{3}\right\rfloor}
   (-1)^{k}
   c_{n-1,\left\lfloor \frac{n-1-3k}{2}\right\rfloor}
   =1, \\
   \sum_{k= 0}^{\left\lfloor \frac{m}{2}\right\rfloor}
   \frac{1}{2}\left((-1)^{\left\lfloor \frac{m-2k}{5}\right\rfloor}-(-1)^{\left\lfloor \frac{m-2k-2}{5}\right\rfloor}\right)
   c_{m+1,k}
   &=
   \sum_{k= 0}^{\left\lfloor \frac{n-1}{5}\right\rfloor}
   (-1)^{k}
   c_{n,\left\lfloor \frac{n-1-5k}{2}\right\rfloor}
   =
   F_{m+1}  
\end{align}
and 
\begin{align}
\sum_{k=0}^{\left\lfloor \frac{n}{2}\right\rfloor}
   (-1)^{k}\binom{n-k}{k}   
   &=\begin{cases}
   1 & (n\equiv 0,1\,(\mathrm{mod}\,6)) \\
   -1 & (n\equiv 3,4\,(\mathrm{mod}\,6)) \\
   0 & ({\rm{others}})
   \end{cases}, \\
\sum_{k=0}^{\left\lfloor \frac{n}{2}\right\rfloor}
   (-1)^{k}\binom{n-k+1}{k}
   F_{n-2k+1}
   &=\begin{cases}
   1 & (n\equiv 0,1\,(\mathrm{mod}\,10)) \\
   -1 & (n\equiv 5,6\,(\mathrm{mod}\,10)) \\
   0 & ({\rm{others}})
   \end{cases}. 
\end{align}
\end{cor}
The formula (\ref{eq:Andrews Fibonacci formula}) was given by Andrews \cite{A}. 
\begin{cor}
\label{thm:new Lucas formulas}
For any non-negative integer $m$, we have
\begin{align}
\label{eq:Lucas cor formula1-1}
\frac{3}{2}\sum_{k=-\left\lfloor \frac{m}{3}\right\rfloor}^{\left\lfloor \frac{m}{3}\right\rfloor}
   \binom{2m}{m-3k}
   &=
   2^{2m-1}+1, \\
   \label{eq:Lucas cor formula1-2}
\frac{3}{2}\sum_{k=-\left\lfloor \frac{m+2}{3}\right\rfloor}^{\left\lfloor \frac{m-1}{3}\right\rfloor}
   \binom{2m+1}{m-3k-1}
   &=
3\sum_{k=0}^{\left\lfloor \frac{m-1}{3}\right\rfloor}
   \binom{2m+1}{m-3k-1}
   =
   4^{m}-1, \\
L_{2m}
   &=
   -2^{2m-1}
   +\frac{5}{2}\sum_{k=-\left\lfloor \frac{m}{5}\right\rfloor}^{\left\lfloor \frac{m}{5}\right\rfloor}
   \binom{2m}{m-5k}, \\
L_{2m+1}
   &=
   4^{m}
   -\frac{5}{2}\sum_{k=-\left\lfloor \frac{m+3}{5}\right\rfloor}^{\left\lfloor \frac{m-2}{5}\right\rfloor}
   \binom{2m+1}{m-5k-2}.
\end{align}
For any positive integer $n$, we have
\begin{align}
2\sum_{k= 0}^{\lfloor \frac{n+1}{2}\rfloor}\binom{2k-n-1}{k}
   -\sum_{k= 0}^{\lfloor \frac{n}{2}\rfloor}\binom{2k-n}{k}
   &=
   \begin{cases}
   (-1)^{n-1} & (3\not| \,\,\,n) \\
   (-1)^{n}2 & (3\mid n)
   \end{cases}, \\
2\sum_{k= 0}^{\lfloor \frac{n+1}{2}\rfloor}\binom{2k-n-1}{k}
   L_{n-2k}
   -\sum_{k= 0}^{\lfloor \frac{n}{2}\rfloor}\binom{2k-n}{k}L_{n-2k}
   &=
   \begin{cases}
   (-1)^{n-1} & (5\not| \,\,\,n) \\
   (-1)^{n}4 & (5\mid n)
   \end{cases}.
\end{align}
\end{cor}

\appendix

\section{Some congruence relations for $F_{n}^{(r)}$ and $L_{n}^{(r)}$}
All of the results so far have been obtained as specializations of Theorem \ref{thm:the first kind} and Theorem \ref{thm:the second kind}, but in this section, we mention some properties for $F_{n}^{(r)}$ and $L_{n}^{(r)}$ that can be obtained independently of Theorem \ref{thm:the first kind} and Theorem \ref{thm:the second kind}. 
\begin{thm}
Let $p:=2r+1$ be a prime number. 
If $q$ is a odd prime number such that $q\equiv \pm 1\,\,(\mathrm{mod}\,p)$, then 
\begin{align}
\label{eq:cong F}
F_{n+q-1}^{(r)}&\equiv F_{n}^{(r)}\,\,(\mathrm{mod}\,q), \\
\label{eq:cong L}
L_{n+q-1}^{(r)}&\equiv L_{n}^{(r)}\,\,(\mathrm{mod}\,q).
\end{align}
In particular, for any non-negative integer $k$ we have
\begin{align}
\label{eq:A3}
F_{k(q-1)}&\equiv 0 \quad (\mathrm{mod}\,q), \\
F_{k(q-1)+1}&\equiv F_{k(q-1)+2}\equiv 1 \quad (\mathrm{mod}\,q), \\
L_{k(q-1)}&\equiv \frac{p-1}{2} \quad (\mathrm{mod}\,q), \\
L_{k(q-1)+1}&\equiv 1 \quad (\mathrm{mod}\,q), \\
\label{eq:A7}
L_{k(q-1)+2}&\equiv p-2 \quad (\mathrm{mod}\,q). 
\end{align}
\end{thm}
\begin{proof}
Let $\mathbb{F}_{q}$ be a finite field order $q$ and $\zeta_{p}$ be a primitive $p$-th root of unity. 
Since the both side of (\ref{eq:cong F}) and (\ref{eq:cong L}) are integers, it is enough to show that the equalities in $\mathbb{F}_{q}[\zeta_{p}]$. 
For any integer $i$ a simple calculation shows that 
\begin{align*}
(\zeta_{p}^{i}+\zeta_{p}^{-i})^{q}
   =
   \zeta_{p}^{qi}+\zeta_{p}^{-qi}
   =
   \zeta_{p}^{i}+\zeta_{p}^{-i} \quad \text{(in $\mathbb{F}_{q}[\zeta_{p}]$)}.
\end{align*}
Hence by the definition of $L_{n}^{r}$ we obtain
\begin{align}
L_{n+q-1}^{(r)}
   &=
   \sum_{j=1}^{r}
      (-\zeta_{p}^{j}-\zeta_{p}^{-j})^{n-1}
      (-\zeta_{p}^{j}-\zeta_{p}^{-j})^{q} \nonumber \\
   &=
   \sum_{j=1}^{r}
      (-\zeta_{p}^{j}-\zeta_{p}^{-j})^{n-1}
      (-\zeta_{p}^{j}-\zeta_{p}^{-j}) \nonumber \\
   &=
   L_{n} \quad \text{(in $\mathbb{F}_{q}[\zeta_{p}]$)}. \nonumber 
\end{align}

To prove (\ref{eq:cong F}), we need the discriminant of $-\zeta_{p}-\zeta_{p}^{-1}$ \cite{L} Theorem 3.8
$$
\det{\left((-\zeta _{p}^{r-i}-\zeta _{p}^{-(r-i)})^{r-j}\right)}_{i,j=1,\ldots r}^{2}
   =
      p^{\frac{p-3}{2}}.
$$
From this evaluation, we have
\begin{equation}
\label{eq:discrim formula}
\det{\left((-\zeta _{p}^{r-i}-\zeta _{p}^{-(r-i)})^{r-j}\right)}_{i,j=1,\ldots r}
   =
   p^{\frac{p-3}{4}}
   =
   p^{\frac{r-1}{2}}.
\end{equation}
We point out even if $r$ is even then $p^{\frac{r-1}{2}} \in \mathbb{F}_{p}$ by the first supplement to quadratic reciprocity. 
Thus we have
\begin{align*}
& p^{\frac{p-3}{4}}F_{n+q}^{(r)} \nonumber \\
   &=
   \det
   \begin{pmatrix}
   (-\zeta _{p}^{r-1}-\zeta _{p}^{-(r-1)})^{n+q-1+r-1} & (-\zeta _{p}^{r-2}-\zeta _{p}^{-(r-2)})^{n+q-1+r-1} & \cdots & (-\zeta _{p}-\zeta _{p}^{-1})^{n+q-1+r-1} \\
   (-\zeta _{p}^{r-1}-\zeta _{p}^{-(r-1)})^{r-2} & (-\zeta _{p}^{r-2}-\zeta _{p}^{-(r-2)})^{r-2} & \cdots & (-\zeta _{p}-\zeta _{p}^{-1})^{r-2} \\
   \vdots & \vdots & \ddots & \vdots \\
   (-\zeta _{p}^{r-1}-\zeta _{p}^{-(r-1)}) & (-\zeta _{p}^{r-2}-\zeta _{p}^{-(r-2)}) & \cdots & (-\zeta _{p}-\zeta _{p}^{-1}) \\
   1 & 1 & \cdots & 1    
   \end{pmatrix} \\
   &=
   \det
   \begin{pmatrix}
   (-\zeta _{p}^{r-1}-\zeta _{p}^{-(r-1)})^{n+r-1} & (-\zeta _{p}^{r-2}-\zeta _{p}^{-(r-2)})^{n+r-1} & \cdots & (-\zeta _{p}-\zeta _{p}^{-1})^{n+r-1} \\
   (-\zeta _{p}^{r-1}-\zeta _{p}^{-(r-1)})^{r-2} & (-\zeta _{p}^{r-2}-\zeta _{p}^{-(r-2)})^{r-2} & \cdots & (-\zeta _{p}-\zeta _{p}^{-1})^{r-2} \\
   \vdots & \vdots & \ddots & \vdots \\
   (-\zeta _{p}^{r-1}-\zeta _{p}^{-(r-1)}) & (-\zeta _{p}^{r-2}-\zeta _{p}^{-(r-2)}) & \cdots & (-\zeta _{p}-\zeta _{p}^{-1}) \\
   1 & 1 & \cdots & 1    
   \end{pmatrix} \\
   &=
   p^{\frac{p-3}{4}}F_{n+1}^{(r)} \quad \text{(in $\mathbb{F}_{q}[\zeta_{p}]$)}.
\end{align*}
Here the first equality follows from (\ref{eq:Schur and complete homogeneous}). 
Finally, since $p$ does not divide $q$, we obtain (\ref{eq:cong F}).

The formulas (\ref{eq:A3}) - (\ref{eq:A7}) follow from (\ref{eq:cong F}), (\ref{eq:cong L}) and Corollary \ref{thm:initial conditions} immediately. 
\end{proof}

\section{Other formulas for $F_{n}^{(r)}$ and $L_{n}^{(r)}$ from symmetric polynomials}
Since the sequences $\{F_{n}^{(r)}\}_{n}$ and $\{L_{n}^{(r)}\}_{n}$ are special values of $h_{n}^{(r)}(\mathbf{z})$ and $p_{n}^{(r)}(\mathbf{z})$ respectively, various formulas for $F_{n}^{(r)}$ and $L_{n}^{(r)}$ are derived immediately from specializations of some formulas for symmetric polynomials \cite{M}. 
In this section, we list some typical formulas obtained from symmetric polynomials.\\
\underline{Generating functions}
\begin{align}
\label{eq:gen fnc F}
   \sum_{n\geq 0}F_{n}^{(r)}u^{n}
   &=
\frac{1}{\sum_{m=0}^{\left\lfloor \frac{r}{2}\right\rfloor }
   (-1)^{m}\binom{r-m}{m}u^{2m}
   -\sum_{m=0}^{\left\lfloor \frac{r-1}{2}\right\rfloor }
   (-1)^{m}\binom{r-1-m}{m}u^{2m+1}}, \\
\label{eq:gen fnc L}
   \sum_{n\geq 0}L_{n+1}^{(r)}u^{n}
   &=
\frac{\sum_{m=0}^{\left\lfloor \frac{r-1}{2}\right\rfloor }
   (-1)^{m}(2m+1)\binom{r-1-m}{m}u^{2m}-\sum_{m=0}^{\left\lfloor \frac{r}{2}\right\rfloor }
   (-1)^{m}2m\binom{r-m}{m}u^{2m-1}}{\sum_{m=0}^{\left\lfloor \frac{r}{2}\right\rfloor }
   (-1)^{m}\binom{r-m}{m}u^{2m}
   -\sum_{m=0}^{\left\lfloor \frac{r-1}{2}\right\rfloor }
   (-1)^{m}\binom{r-1-m}{m}u^{2m+1}}.
\end{align}
Generating functions (\ref{eq:gen fnc F}) and (\ref{eq:gen fnc L}) are obtained by substituting (\ref{eq:3rd result2}) into (2.5) and (2.10) in \cite{M}.\\
\underline{Determinant formulas} For convenience, put 
$$
\alpha_{r,j}
   :=
   -2\cos{\left(\frac{2\pi (r+1-j)}{2r+1}\right)}, \quad (j=1,\ldots,r)
$$
and 
$$
C_{n}^{(r)}
   :=
   e_{n}^{(r)}(-\boldsymbol{\zeta}^{+\iota }-\boldsymbol{\zeta}^{-\iota })
   =
   \begin{cases}
   \binom{n-r-1}{\left\lfloor \frac{n}{2}\right\rfloor}=(-1)^{\left\lfloor \frac{n}{2}\right\rfloor}\binom{r-\left\lfloor \frac{n+1}{2}\right\rfloor}{\left\lfloor \frac{n}{2}\right\rfloor} & (n=0,1,\ldots,r)\\
   0 & (n\not=0,1,\ldots,r)
   \end{cases}.
$$
From (\ref{eq:Schur and complete homogeneous}), (\ref{eq:discrim formula}) and the determinant formulas on p28 of \cite{M}, we obtain the following determinant formulas for $F_{n}^{(r)}$ and $L_{n}^{(r)}$. 
\begin{align}
F_{n+1}^{(r)}
   &=
   \frac{1}{(2r+1)^{\frac{r-1}{2}}}
   \det
   \begin{pmatrix}
   \alpha_{r,1}^{n+r-1} & \alpha_{r,2}^{n+r-1} & \cdots & \alpha_{r,r-1}^{n+r-1} & \alpha_{r,r}^{n+r-1} \\
   \alpha_{r,1}^{r-2} & \alpha_{r,2}^{r-2} & \cdots & \alpha_{r,r-1}^{r-2} & \alpha_{r,r}^{r-2} \\
   \vdots & \vdots & \ddots & \vdots & \vdots \\
   \alpha_{r,1} & \alpha_{r,2} & \cdots & \alpha_{r,r-1} & \alpha_{r,r} \\
   1 & 1 & \cdots & 1 & 1 
   \end{pmatrix} \\
   &=
   \det\left(C_{1-i+j}^{(r)}\right)_{1\leq i,j\leq n} \\
   &=
   \frac{1}{n!}
   \det
   \begin{pmatrix}
   L_{1}^{(r)} & -1 & 0 & \cdots & 0 & 0 \\
   L_{2}^{(r)} & L_{1}^{(r)} & -2 & \cdots & 0 & 0 \\
   L_{3}^{(r)} & L_{2}^{(r)} & L_{1}^{(r)} & \ddots & 0 & 0 \\
   \vdots & \vdots & \vdots & \ddots & \ddots & \vdots \\
   L_{n-1}^{(r)} & L_{n-2}^{(r)} & L_{n-3}^{(r)} & \cdots & L_{1}^{(r)} & -n+1 \\   
   L_{n}^{(r)} & L_{n-1}^{(r)} & L_{n-2}^{(r)} & \cdots & L_{2}^{(r)} & L_{1}^{(r)} 
   \end{pmatrix}, \\   
L_{n}^{(r)}
   &=
   \det
   \begin{pmatrix}
   C_{1}^{(r)} & 1 & 0 & \cdots & 0 & 0 \\
   2C_{2}^{(r)} & C_{1}^{(r)} & 1 & \cdots & 0 & 0 \\
   3C_{3}^{(r)} & C_{2}^{(r)} & C_{1}^{(r)} & \ddots & 0 & 0 \\
   \vdots & \vdots & \vdots & \ddots & \ddots & \vdots \\
   (n-1)C_{n-1}^{(r)} & C_{n-2}^{(r)} & C_{n-3}^{(r)} & \cdots & C_{1}^{(r)} & 1 \\
   nC_{n}^{(r)} & C_{n-1}^{(r)} & C_{n-2}^{(r)} & \cdots & C_{2}^{(r)} & C_{1}^{(r)} \\ 
   \end{pmatrix} \\
   &=
   (-1)^{n-1}
   \det
   \begin{pmatrix}
   F_{2}^{(r)} & F_{1}^{(r)} & 0 & \cdots & 0 & 0 \\
   2F_{3}^{(r)} & F_{2}^{(r)} & F_{1}^{(r)} & \cdots & 0 & 0 \\
   3F_{4}^{(r)} & F_{3}^{(r)} & F_{2}^{(r)} & \ddots & 0 & 0 \\
   \vdots & \vdots & \vdots & \ddots & \ddots & \vdots \\
   (n-1)F_{n}^{(r)} & F_{n-1}^{(r)} & F_{n-2}^{(r)} & \cdots & F_{2}^{(r)} & F_{1}^{(r)} \\ 
   nF_{n+1}^{(r)} & F_{n}^{(r)} & F_{n-1}^{(r)} & \cdots & F_{3}^{(r)} & F_{2}^{(r)} 
   \end{pmatrix}, \\
C_{n}^{(r)}
   &=
   \det\left(F_{2-i+j}^{(r)}\right)_{1\leq i,j\leq n} \\
   &=
   \frac{1}{n!}
   \det
   \begin{pmatrix}
   L_{1}^{(r)} & 1 & 0 & \cdots & 0 & 0 \\
   L_{2}^{(r)} & L_{1}^{(r)} & 2 & \cdots & 0 & 0 \\
   L_{3}^{(r)} & L_{2}^{(r)} & L_{1}^{(r)} & \ddots & 0 & 0 \\
   \vdots & \vdots & \vdots & \ddots & \ddots & \vdots \\
   L_{n-1}^{(r)} & L_{n-2}^{(r)} & L_{n-3}^{(r)} & \cdots & L_{1}^{(r)} & n-1 \\   
   L_{n}^{(r)} & L_{n-1}^{(r)} & L_{n-2}^{(r)} & \cdots & L_{2}^{(r)} & L_{1}^{(r)} 
   \end{pmatrix}.
\end{align}
\underline{Some relations} 
For any partition $\lambda $ let $z_{\lambda }$ denote the product 
$$
z_{\lambda }:=\prod_{i\geq 1}i^{m_{i}}m_{i}!
$$
where $m_{i}=m_{i}(\lambda )$ is the number of parts of $\lambda $ equal to $i$. 
Then we have
\begin{align}
F_{n+1}^{(r)}
   &=
   \frac{1}{n}\sum_{i=1}^{n}
      L_{i}^{(r)}F_{n+1-i}^{(r)} \\
   &=
   \sum_{|\lambda |=n}
   \frac{L_{\lambda _{1}}^{(r)}\cdots L_{\lambda _{r}}^{(r)}}{z_{\lambda }}, \\
C_{n}^{(r)}
   &=
   \sum_{|\lambda |=n}
   (-1)^{n-r}
   \frac{L_{\lambda _{1}}^{(r)}\cdots L_{\lambda _{r}}^{(r)}}{z_{\lambda }}
\end{align}
where $\lambda =(\lambda _{1},\ldots,\lambda _{r})$ run over partitions 
and $|\lambda |$ denote the sum of the parts
$$
|\lambda |:=\lambda _{1}+\cdots +\lambda _{r}.
$$
These formulas follow from (2.11) and (2.14${}^{\prime}$) in \cite{M}.


\bibliographystyle{amsplain}

\noindent 
Department of Mathematics, Graduate School of Science, Kobe University, \\
1-1, Rokkodai, Nada-ku, Kobe, 657-8501, JAPAN\\
E-mail: g-shibukawa@math.kobe-u.ac.jp

\end{document}